\documentclass[12pt]{amsart}
\usepackage{latexsym}
\usepackage{a4}
\usepackage{amssymb}
\usepackage{hyperref}

\newtheorem{Thm}{Theorem}[section] 
 
\newtheorem{Lem}[Thm]{Lemma} 
\newtheorem{Cor}[Thm]{Corollary} 
\theoremstyle{definition}

\newtheorem{Def}[Thm]{Definition}
\newtheorem{Rem}[Thm]{Remark}
\numberwithin{equation}{section}

\newcommand{\Clo}{\mathrm{Clo}}
\newcommand{\Cid}{\mathrm{Clg}}
\newcommand{\Pol}{\mathrm{Pol}}
\newcommand{\ZZ}{{\mathbb{Z}}}
\newcommand{\NN}{{\mathbb{N}}}
\newcommand{\FF}{{\mathbb{F}}}
\DeclareMathAlphabet\mathbfsl {T1}{cmr}{bx}{it}
\title[Closed sets of functions]{Closed sets of finitary functions between finite fields of coprime order}
\author{Stefano Fioravanti}

\address{Stefano Fioravanti,
Institut f\"ur Algebra,
Johannes Kepler Universit\"at Linz,
4040 Linz,
Austria}
\email{\tt stefano.fioravanti@jku.at}
\subjclass[2018]{08A40}
\urladdr{http://www.jku.at/algebra}
\thanks{Supported by the Austrian Science Fund (FWF):P29931.}
\keywords{Clonoids, Clones}
\date{\today}
\begin{document}

\begin{abstract}

We investigate the finitary functions from a finite field $\mathbb{F}_q$ to the finite field $\mathbb{F}_p$, where $p$ and $q$ are powers of different primes. An $(\mathbb{F}_p,\mathbb{F}_q)$-linearly closed clonoid is a subset of these functions which is closed under composition from the right and from the left with linear mappings. 
	
We give a characterization of these subsets of functions through the invariant subspaces of the vector space $\mathbb{F}_p^{\mathbb{F}_q\backslash\{0\}}$ with respect to a certain linear transformation with minimal polynomial $x^{q-1} - 1$. Furthermore we prove that each of these subsets of functions is generated by one unary function. 
	
\end{abstract}

\maketitle

\section{Introduction}

The problem of characterizing sets of functions that satisfy some closure properties plays an increasingly important role in General Algebra. E. Post's characterization of all clones on a two-element set \cite{Pos.TTVI} is a foundational result in this field, which was developed further, e. g., in \cite{Ros.MCOA,PK.FUR,Sze.CIUA,Leh.CCOF}. Starting from \cite{BJK.TCOC}, clones are used to study the complexity of certain constraint satisfaction problems (CSPs).
 
The aim of this paper is to describe sets of functions from $\mathbb{F}_q$ to $\mathbb{F}_p$ that are closed with linear mappings from the left and from the right, in the case $p$ and $q$ are powers of distinct primes. We are dealing with sets of functions with different domains and codomains; such sets are investigated, e. g., in \cite{AM.FGEC} and are called clonoids. Let $\mathbf{B}$ be an algebra, and let $A$ be a non-empty set. For a subset $C$ of $\bigcup_{n \in \NN} B^{A^n}$ and $k\in \NN$, we let $C^{[k]} :=C \cap B^{A^k}$. According to Definition $4.1$ of \cite{AM.FGEC} we call $C$ a \emph{clonoid} with source set $A$ and target algebra $\mathbf{B}$ if
\begin{center}
	\begin{enumerate}
		\item [(1)] for all $k \in \NN$: $C^{[k]}$ is a subuniverse of $\mathbf{B}^{A^k}$, and
		\item [(2)] for all $k,n \in \NN$, for all $(i_1,\dots,i_k) \in \{1,\dots,n\}^k$, and for all $c \in C^{[k]}$, the function $c' : A^n \to B$ with $c'(a_1,\dots,a_n) := c(a_{i_1},\dots,a_{i_k})$ lies in $C^{[n]}$.
	\end{enumerate}
\end{center}

By $(1)$ every clonoid is closed under composition with operations of $\mathbf{B}$ on the left. In particular we are interested in those clonoids whose target algebra is the vector space $\mathbb{F}_p$ which are closed under composition with linear mappings also from the right side.

\begin{Def}
	\theoremstyle{definition}
	\label{DefClo}
	Let $p$ and $q$ be powers of different primes, and let $\mathbb{F}_p$ and $\mathbb{F}_q$ be two fields of orders $p$ and $q$. An \emph{$(\mathbb{F}_p,\mathbb{F}_q)$-linearly closed clonoid} is a non-empty subset $C$ of $\bigcup_{n \in \mathbb{N}} \mathbb{F}^{\mathbb{F}^n_q}_p$ with the following properties:
	
	\begin{enumerate}
		\item[(1)] for all $n \in \NN$, $f,g \in C^{[n]}$ and $a, b \in \mathbb{F}_p$:
		
		\begin{center}
			$af + bg \in C^{[n]}$;
		\end{center}
		
		\item[(2)] for all $m,n \in \NN$, $f \in C^{[m]}$ and $A \in \mathbb{F}^{m \times n}_q$:
		
		\begin{center}
			$g: (x_1,\dots,x_n) \mapsto f(A\cdot (x_1,\dots,x_n)^t)$ is in $C^{[n]}$.
		\end{center}
		
	\end{enumerate}
\end{Def}

Clonoids are of interest since they naturally arise in the study of promise constraint satisfaction problems (PCSPs). These problems are investigated,  e. g., in \cite{BG.PCSS}, and recently in \cite{BKO.AATP} clonoid theory has been used to give an algebraic approach to PCSPs. Moreover, a description of the set of all $(\mathbb{Z}_p,\mathbb{Z}_q)$-linearly closed clonoids, where $p$ and $q$ are distinct primes, is a useful tool to investigate (polynomial) clones on $\ZZ_p \times \ZZ_q$ or to represent polynomial functions of semidirect products of groups. In \cite{Kre.CFSO} S. Kreinecker characterized linearly closed clonoids on $\mathbb{Z}_p$, where $p$ is a prime, and found a description of all clones on $\mathbb{Z}_p$ that contain the addition, all iterative algebras on $\ZZ_p$ which are closed under composition with the clone generated by $+$ from both sides, and proved that there are infinitely many non-finitely generated clones above $\Clo(\mathbb{Z}_p \times \mathbb{Z}_p , +)$ for $p > 2$.

Our main result (Theorem \ref{ThmShapeL1}) provides a complete description of  the lattice of all $(\mathbb{F}_p,\mathbb{F}_q)$-linearly closed clonoids, where $p$ and $q$ are powers of different primes. First, an important observation is that each such clonoid is generated by its subset of unary members (Theorem \ref{Th1}). We can say even more about the generators of an $(\mathbb{F}_p,\mathbb{F}_q)$-linearly closed clonoid.

\begin{Thm}
	\label{Thm14}
	Let $p$ and $q$ be powers of different primes. Then every $(\mathbb{F}_p,\mathbb{F}_q)$-linearly closed clonoid is generated by one unary function.
\end{Thm}

The proof of this result is given in Section \ref{TheLattice}. With Theorem \ref{Th1} and the characterization of the invariant subspaces lattice of a cyclic linear transformation over a finite-dimensional vector space in \cite{BF.TISL}, we obtain a description of the lattice of all $(\mathbb{F}_p,\mathbb{F}_q)$-linearly closed clonoids as a direct product of chains (Section \ref{TheLattice}). 

The structure of the lattice of all $(\mathbb{F}_p,\mathbb{F}_q)$-linearly closed clonoids depends on the prime factorization of the polynomial $g = x^{q-1} -1$ in $\mathbb{F}_p[x]$. Once this factorization is known, it is easy to find this lattice. Let us denote by $\mathbf{2}$ the two-element chain and, in general, by $\mathbf{C}_k$ the chain with $k$ elements. Moreover, we denote by $\mathcal{L}(p,q)$ the lattice of all $(\mathbb{F}_p,\mathbb{F}_q)$-linearly closed clonoids.

\begin{Thm}
	\label{ThmShapeL1}
	Let $p$ and $q$ be powers of different primes. Let $\prod_{i =1}^n p_i^{k_i}$ be the factorization of the polynomial $g = x^{q-1} -1$ in $\mathbb{F}_p[x]$ into its irreducible divisors. Then the number of distinct $(\mathbb{F}_p,\mathbb{F}_q)$-linearly closed clonoids is $2\prod_{i =1}^n(k_i +1)$ and the lattice of all $(\mathbb{F}_p,\mathbb{F}_q)$-linearly closed clonoids, $\mathcal{L}(p,q)$, is isomorphic to
		
	\begin{center}
		$\mathbf{2} \times \prod_{i =1}^n \mathbf{C}_{k_i+1}$.
	\end{center}
\end{Thm}

\section{Preliminaries and notations}\label{Preliminaries}

We use boldface letters for vectors, e. g., $\mathbfsl{u} = (u_1,\dots,u_n)$ for some $n \in \NN$. Moreover, we will use $\langle\mathbfsl{v}, \mathbfsl{u}\rangle$ for the scalar product of the vectors $\mathbfsl{v}$ and $\mathbfsl{u}$.

We write $\Cid(S)$ for the $(\mathbb{F}_p,\mathbb{F}_q)$-linearly closed clonoid generated by a set of functions $S$. 
The $(\mathbb{F}_p,\mathbb{F}_q)$-linearly closed clonoids form a lattice with the intersection as meet and the $(\mathbb{F}_p,\mathbb{F}_q)$-linearly closed clonoid generated by the union as join. The top element of the lattice is the $(\mathbb{F}_p,\mathbb{F}_q)$-linearly closed clonoid of all functions and the bottom element consists of only the constant zero functions. Let $f$ be an $n$-ary function from a group $\mathbf{G}_1$ to a group $\mathbf{G}_2$. We say that $f$ is \emph{$0$-preserving} if $f(0_{\mathbf{G}_1},\dots,0_{\mathbf{G}_1})	= 0_{\mathbf{G}_2}$. 

As examples of non-trivial $(\mathbb{F}_p,\mathbb{F}_q)$-linearly closed clonoids we can see that the set of all $0$-preserving finitary functions from $\mathbb{F}_q$ to $\mathbb{F}_p$ forms an $(\mathbb{F}_p,\mathbb{F}_q)$-linearly closed clonoid and that the following set of functions forms an $(\mathbb{F}_p,\mathbb{F}_q)$-linearly closed clonoid.

\begin{Def}
	\theoremstyle{definition}
	\label{Def6}
	Let $p$ and $q$ be powers of primes and let $f$ be a function from $\mathbb{F}^n_q$ to $\mathbb{F}_p$. The function $f$ is a \emph{star function} if and only if for every vector $\mathbfsl{w} \in \mathbb{F}^n_q$ there exists $k \in \mathbb{F}_p$ such that for every $\lambda \in \mathbb{F}_q \backslash \{0\}$:
	
	\begin{center}
		$f(\lambda \mathbfsl{w}) = k$.
	\end{center}
	
\end{Def}

It is easy to see that the star functions form an $(\mathbb{F}_p,\mathbb{F}_q)$-linearly closed clonoid for every $p$ and $q$ and they represent an instance of the nice behaviour that the $(\mathbb{F}_p,\mathbb{F}_q)$-linearly closed clonoids have in relation to the lines of the space $\mathbb{F}_q^n$. Indeed, the composition with scalar multiplications from the right hand side can be used to permute the values that functions of $(\mathbb{F}_p,\mathbb{F}_q)$-linearly closed clonoid have in lines that pass through the origin.

\section{Preliminaries from linear algebra}\label{PreliminariesLA}

In this section we review some concepts of linear algebra that we need in order to find a description of the lattice of all $(\mathbb{F}_p,\mathbb{F}_q)$-linearly closed clonoids, where $p$ and $q$ are powers of distinct primes.
We recall that a \emph{T-invariant subspace} of a linear operator $T$ of a vector space $V$ is a subspace $W$ of $V$ that is preserved by $T$; that is, $T(W) \subseteq W$.
Let $S$ be a set of linear operators of a vector space $V$. We can consider the $S$-invariant subspaces lattice of $V$ and we denote it by $\mathcal{L}(S)$.
	
In Section \ref{TheLattice} we will see that the problem to find the structure of the lattice of all $(\mathbb{F}_p,\mathbb{F}_q)$-linearly closed clonoids can be reduced to the problem to find all $T$-invariant subspaces of the vector space $\mathbb{F}_p^{\mathbb{F}_q\backslash\{0\}}$, where $T$ is certain linear transformation that permutes the components of $\mathbb{F}_p^{\mathbb{F}_q\backslash\{0\}}$.
In \cite{BF.TISL} the structure of the invariant subspaces lattice of a linear transformation on a finite-dimensional vector space over an arbitrary field has been studied, and in \cite{Fri.TNOI} the number of invariant subspaces of a finite vector space with respect to a linear operator is determined.
 
Let $T$ be a linear transformation on a finite-dimensional vector space $V$ over a field $\mathbb{K}$ and let $g$ be the minimal polynomial of $T$. We call $T$ \emph{primary} if $g = f^c$ for some irreducible polynomial $f$ and some positive integer $c$. We know from \cite[Theorem $1$]{BF.TISL} that, with the prime factorization of $g = \prod_{i = 1}^sp_i^{k_i}$ over $\mathbb{K}[\mathbfsl{x}]$, we can split the vector space $V$ into what is called its \emph{primary decomposition}:

\begin{center}
	$V =\bigoplus_{i = 1}^s V_i$,
\end{center}
and $V_i = ker(p_i(T)^{k_i})$ are called the \emph{primary components} of $V$. According to \cite{BF.TISL}, the lattice $\mathcal{L}(T)$ of the  $T$-invariant subspaces of $V$ is a direct product of the lattices $\mathcal{L}(T_i)$, where $T_i = T|_{V_i}$. Thus:

\begin{center}
	$\mathcal{L}(T) = \prod_{i=1}^s\mathcal{L}(T_i)$.
\end{center}

\begin{Def}
	\theoremstyle{definition}
	\label{DefCyclic}
	Let $V$ be a vector space, let $M$ be a subspace of $V$, and let $T$ be a linear transformation. If $M$ is generated by $\{x, Tx, T^2x,\dots \}$ for some $x \in V$, then $T$ is called \emph{cyclic}, $M$ is called a $T$-\emph{cyclic subspace}, and $x$ is called a $T$-\emph{cyclic vector} for $M$.
\end{Def}

 \begin{Rem}
 	\label{Rem}
 	Let $V$	be a finite dimensional vector space over a field $\mathbb{K}$ and let $T: V \rightarrow V$ be a linear operator such that $V$ is $T$-cyclic. Then every $T$-invariant subspace of $V$ is $T$-cyclic.
 \end{Rem}
 
 \begin{proof}
 	Let $W = \langle \mathbfsl{w}_1,\dots,\mathbfsl{w}_n\rangle$ be a $T$-invariant subspace of $V$ and let $\mathbfsl{v}$ be a $T$-cyclic vector of $V$. Then there exist $p_1,\dots,p_n \in \mathbb{K}[x]$ such that $\mathbfsl{w}_i = p_i(T)\mathbfsl{v}$ for all $i \in \{1, \dots,n\}$. Let $d = gcd(p_1,\dots,p_n)$. Then $W =\{q(T)(d(T)\mathbfsl{v}) \mid q \in \mathbb{K}[x]\}$. Thus $W$ is $T$-cyclic with  $T$-cyclic vector $d(T)\mathbfsl{v}$.
 \end{proof}

In \cite[Lemma $2$]{BF.TISL} it is proved that $\mathcal{L}(T)$ is a chain if and only if $T$ is cyclic and primary. In particular they show that if the minimal polynomial of $T$ is $g = f^n$, with $f$ irreducible, then:

\begin{center}

$\mathcal{L}(T) = \{ker(f(T)^ k)\mid k \in \{0, 1, . . . ,n\}\}$.

\end{center}

\section{Generators of  $(\mathbb{F}_p,\mathbb{F}_q)$-linearly closed clonoids}\label{AllGen}

In this section our aim is to find a set of unary generators of an $(\mathbb{F}_p,\mathbb{F}_q)$-linearly closed clonoid. In general we will see that it is the unary part of an $(\mathbb{F}_p,\mathbb{F}_q)$-linearly closed clonoid that determines the clonoid. To this end we shall show the following Lemma.

\begin{Lem}
\label{Lem1}
Let $f,g: \mathbb{F}_q^n \to \mathbb{F}_p$ be functions such that there exists $\mathbfsl{b} \in \mathbb{F}_q^n \backslash$ $\{(0,\dots,0)\}$ with $f(\lambda\mathbfsl{b}) =  g(\lambda(1,0,\dots,0))$ for all $\lambda \in \mathbb{F}_q$ and $f(\mathbfsl{x}) = g(\mathbfsl{y}) = 0$ for all $\mathbfsl{x} \in \mathbb{F}_q^n \backslash \{\lambda\mathbfsl{b} \mid \lambda \in \mathbb{F}_q\}$ and $\mathbfsl{y} \in \mathbb{F}_q^n\backslash \{\lambda(1,0,\dots,0) \mid \lambda \in \mathbb{F}_q\}$. Then $f \in \Cid(\{g\})$.
\end{Lem}

\begin{proof}
Let $n \in \NN$ and $\mathbfsl{b} = (b_1,\dots,b_n) \in \mathbb{F}_q^n \backslash\{(0,\dots,0)\}$. Let $1 \leq i \leq n$ be such that $b_i \not= 0$ and let $f,g: \mathbb{F}_q^n \to \mathbb{F}_p$ be functions as in the hypothesis. Moreover, let $L = \{s\mathbfsl{b} \mid s \in \mathbb{F}_q\}$ be the line of the space $\mathbb{F}_q^n$ generated by the vector $\mathbfsl{b}$. Let us consider $\mathbfsl{l}_j \in \mathbb{F}_q^n$ for $1 \leq j \leq n-1$ such that the solutions of the system formed by the equations $(\langle\mathbfsl{l}_j,\mathbfsl{y}\rangle = 0)_{1 \leq j \leq n-1}$ describe the line $L$ of $\mathbb{F}_q^n$.
Then:
\begin{center}
 $g(b_i^{-1}x_i, \langle\mathbfsl{l}_1,\mathbfsl{x}\rangle,\dots,\langle\mathbfsl{l}_{n-1},\mathbfsl{x}\rangle) = f(\mathbfsl{x})$ \ \ \ \ \ \  for all $\mathbfsl{x} = (x_1,\dots,x_n)\in \mathbb{F}_q^n$.
 \end{center}
Hence $g \in \Cid(\{g\})$ and the claim holds.

\end{proof}

In order to show the main theorem of the section we introduce the definition of \emph{Lagrange interpolation functions}, which are functions built to have a value different from zero only in one point, and they can be seen as characteristic functions of a point in the vector space $\mathbb{F}_q^n$ with codomain $\{0,1\} \subseteq \mathbb{F}_p$.

\begin{Def}
	\label{Def9}
	Let $\mathbfsl{a} = (a_1,\dots,a_n) \in \mathbb{F}^n_q$. The \emph{$n$-ary Lagrange interpolation function} $f_{\mathbfsl{a}}$ from $\mathbb{F}_q$ to $\mathbb{F}_p$ is the function defined by:
	
	\begin{center}
		$\begin{array}{cccc}
		f(\mathbfsl{a}) &= &1
		\\f(\mathbfsl{x}) &= &0 &\ \ \ \ \ \text{for} \  \mathbfsl{x} \in \mathbb{F}_q^n \backslash\{\mathbfsl{a}\}.	
		\end{array}$
		
	\end{center}
	
\end{Def}

We are now ready to prove that an $(\mathbb{F}_p,\mathbb{F}_q)$-linearly closed clonoid $C$ is generated by its unary part.

\begin{Thm}
	\label{Th1}	
	Let $p$ and $q$ be powers of different primes. Then every $(\mathbb{F}_p,\mathbb{F}_q)$-linearly closed clonoid $C$ is generated by its unary functions. Thus $C = \Cid(C^{[1]})$.
	
\end{Thm}

\begin{proof}
 The inclusion $\supseteq$ is obvious. For the other inclusion let $C$ be an $(\mathbb{F}_p,\mathbb{F}_q)$-linearly closed clonoid and let $f$ be an $n$-ary function in $C$. In order to prove that $f \in \Cid(C^{[1]})$ we show that $f' = f - f(\mathbfsl{0})$ is in $\Cid(C^{[1]})$, where $f(\mathbfsl{0})$ is the constant $n$-ary function with value $f(\mathbfsl{0})$. This implies the claim because the $n$-ary constant function with value $f(\mathbfsl{0})$ is in $\Cid(C^{[1]})$ by Definition \ref{DefClo}. We can see that $f'$ is a $0$-preserving function of $C$. The strategy is to interpolate  $f'$ in every line passing through the origin. To this end, let $R = \{L_i \mid 1 \leq i \leq (q^{n} - 1)/(q-1) = s\}$ be the set of all $s$ distinct lines of the space $\mathbb{F}_q^n$ that pass through the origin, parametrized by the vectors $\mathbfsl{l}_i \in \mathbb{F}_q^n$ with $i \in \{1,\dots s\} = I$. For all $i \in I$, let $f_{L_i}: \mathbb{F}_q^n \to \mathbb{F}_p$ be defined by:
 
\begin{center}
		$\begin{array}{cccc}
		f_{L_i}(\lambda\mathbfsl{l}_i) &=& f'(\lambda\mathbfsl{l}_i) & \text{for}\ \lambda \in \mathbb{F}_q\\
		f_{L_i}(\mathbfsl{x}) &=& 0 &\ \ \ \ \ \ \ \text{for\ } \mathbfsl{x} \in \mathbb{F}^n_q\backslash \{\lambda\mathbfsl{l}_i \mid \lambda \in \mathbb{F}_q\}.
		\end{array}$
	\end{center}
Since $f'$ is $0$-preserving we can write $f'$ as:\begin{center}
	$$f' = \sum_{i = 1}^{s} f_{L_i}.$$
\end{center}
To prove that $f \in \Cid(C^{[1]})$ it is therefore sufficient to show that $f_{L_i} \in \Cid(C^{[1]})$ for all $L_i \in R$. Let $i \in I$ and let $g: \mathbb{F}_q \to \mathbb{F}_p$ be a function such that $f_{L_i}(x\mathbfsl{l}_i) = g(x) = f'(x\mathbfsl{l}_i)$. Then we prove by induction on the arity $m$ that the function $t_m: \mathbb{F}_q^m \to \mathbb{F}_p$ defined by: 

\begin{center}
	$\begin{array}{ccc}
	t_m(x,0,\dots,0) =& g(x) &\ \ \ \ \text{for all\ } x \in \mathbb{F}_q\\
	t_m(x_1,\dots,x_n) =& 0 &\ \ \ \ \text{for\ } \mathbfsl{x} \in \mathbb{F}^n_q\backslash \{\lambda(1,0,\dots,0) \mid \lambda \in \mathbb{F}_q\},
	\end{array}$
\end{center}
is in $\Cid(C^{[1]})$.

Case $m =1$: if $m = 1$ then $t_1 = g$ is a unary function of $C^{[1]}$. 

Case $m>1$: by the induction hypothesis we know that $t_{m-1} \in \Cid(C^{[1]})$. We define $s_m: \mathbb{F}_q^2 \rightarrow \mathbb{F}_q^m$ by $s_m(i,j) = (i,j,0,\dots,0)$. We denote by $f_{s_m(h,k)}$ the Lagrange interpolation function of the point $s_m(h,k)$ (Definition \ref{Def9}). Let us define the function $r: \mathbb{F}_q^m \rightarrow \mathbb{F}_p$ by:

\begin{equation}\label{eq4}
\begin{split}
r(\mathbfsl{x}) = & \sum_{a \in \mathbb{F}_q} t_{m-1}(x_1 - ax_2,x_3,\dots,x_{m}) - \sum_{a \in \mathbb{F}_q \backslash \{0\}} t_{m-1}(ax_2,x_3,\dots,x_{m})\\
 = &\sum_{a,i,j \in \mathbb{F}_q}g(i)f_{s_m(i+aj,j)}(\mathbfsl{x}) - \sum_{i,j \in \mathbb{F}_q, a \in \mathbb{F}_q\backslash\{0\}}g(i)f_{s_m(j,ia^{-1})}(\mathbfsl{x})\\
 = &\sum_{a,i\in \mathbb{F}_q}g(i)f_{s_m(i,0)}(\mathbfsl{x}) + \sum_{i \in \mathbb{F}_q}g(i)\sum_{j \in \mathbb{F}_q\backslash\{0\}, a \in \mathbb{F}_q}f_{s_m(i+aj,j)}(\mathbfsl{x}) - 
 \\&-\sum_{i \in \mathbb{F}_q}g(i)\sum_{j \in \mathbb{F}_q, a \in \mathbb{F}_q\backslash\{0\}}f_{s_m(j,ia^{-1})}(\mathbfsl{x}),
  \end{split}
 \end{equation} 
for all $\mathbfsl{x} = (x_1,\dots,x_m) \in \mathbb{F}_q^m$. We prove that $r(x_1,\dots,x_m) = q t_{m} (x_1,\dots,x_m)$ for all $(x_1,\dots,x_m) \in \mathbb{F}_q^m$. We only consider the case $x_3 = \cdots = x_n = 0$. 

Case $x_2 = 0$: in this case $r$ evaluates to $q g(x_1)$ for all $x_1 \in \mathbb{F}_q$, as required. 

Case $x_2 \not= 0$: we have $r(\mathbfsl{x}) = \sum_{a \in \mathbb{F}_q} g(x_1 - ax_2 ) - \sum_{a \in \mathbb{F}_q\backslash\{0\}} g(ax_2 ) = \sum_{a \in \mathbb{F}_q} g(a) $ $- \sum_{a \in \mathbb{F}_q\backslash\{0\}} g(a) = g(0) = 0$.

 Because of (\ref{eq4}), we have $r \in \Cid(\{t_{m-1}\}) \subseteq \Cid(C^{[1]})$. Hence we have that $qt_{m} \in \Cid(C^{[1]})$ and thus $t_{m} \in \Cid(C^{[1]})$. This concludes the induction. Thus $t_{n} \in \Cid(C^{[1]})$ and we can see that $f_{L_i}(\lambda\mathbfsl{l}_i) = t_n(\lambda(1,0,\dots,0)) = g(\lambda)$, for all $\lambda \in \mathbb{F}_q$, and $f_{L_i}(\mathbfsl{x}) = t_n(\mathbfsl{y}) = 0$ for all $\mathbfsl{x} \in \mathbb{F}_q^n \backslash \{\lambda\mathbfsl{l}_i \mid \lambda \in \mathbb{F}_q\}$ and $\mathbfsl{y} \in \mathbb{F}_q^n \backslash \{\lambda(1,0,\dots,0) \mid \lambda \in \mathbb{F}_q\}$. By Lemma \ref{Lem1}, $f_{L_i} \in \Cid(\{t_n\}) \subseteq \Cid(C^{[1]})$, which concludes the proof.

\end{proof}

\begin{Thm}
	\label{Thmf1}
	Let $p$ and $q$ be two powers of distinct primes. Then the $(\mathbb{F}_p,\mathbb{F}_q)$-linearly closed clonoid of all $0$-preserving functions is generated by the unary Lagrange interpolation function $f_1$ (Definition $\ref{Def9}$). 
	
\end{Thm}

\begin{proof}
The proof follows directly from Theorem \ref{Th1} since $\Cid(f_1)$ contains every $0$-preserving unary function.
\end{proof}

The following two corollaries of Theorem \ref{Th1} tell us that there are only finitely many distinct $(\mathbb{F}_p,\mathbb{F}_q)$-linearly closed clonoids.

\begin{Cor}
	\label{Cor2}
	Let $p$ and $q$ be powers of different primes. Let $C$ and $D$ be two $(\mathbb{F}_p,\mathbb{F}_q)$-linearly closed clonoids. Then $C = D$ if and only if $C^{[1]} = D^{[1]}$.
\end{Cor}

\begin{Cor}
\label{Th25}
	
Let $p$ and  $q$ be powers of distinct prime numbers. Then every $(\mathbb{F}_p,\mathbb{F}_q)$-linearly closed clonoid has a set of finitely many unary functions as generators, and hence there are only finitely many distinct $(\mathbb{F}_p,\mathbb{F}_q)$-linearly closed clonoids.
	
\end{Cor}	

\section{The lattice of all $(\mathbb{F}_p,\mathbb{F}_q)$-linearly closed clonoids}\label{TheLattice}

In this section we investigate the structure of the lattice $\mathcal{L}(p,q)$ of all $(\mathbb{F}_p,\mathbb{F}_q)$-linearly closed clonoids through a characterization of their unary parts. We call the $(\mathbb{F}_p,\mathbb{F}_q)$-linearly closed clonoids that are composed by only $0$-preserving functions \emph{$0$-preserving $(\mathbb{F}_p,\mathbb{F}_q)$-linearly closed clonoids}. We will see that there is an isomorphism between the sublattice $\mathcal{L}_0(p,q)$ of the $0$-preserving $(\mathbb{F}_p,\mathbb{F}_q)$-linearly closed clonoids and the lattice of subspaces that are invariant under a particular cyclic linear transformation $A(p,q)$ on the vector space $\mathbb{F}_p^{\FF_{q-1}}$.

In order to characterize the lattice of all $(\mathbb{F}_p,\mathbb{F}_q)$-linearly closed clonoids we need the definition of  \emph{monoid ring}.

\begin{Def} Let $\langle M, +\rangle$ be a monoid and let $\langle R, +, \odot\rangle$ be a commutative ring with identity. Let
	\begin{center}
		$S := \{f \in R^M \mid f(a) \not= 0 \text{ for only finitely many } a \in M\}.$
	\end{center}
\end{Def}

We define the \emph{monoid ring} of $M$ over $R$ as the ring $(S, +, \cdot)$. Where $+$ is the point-wise addition of functions and $(\sigma \cdot \rho)(a) := \sum_{b\in M} \sigma(b) \odot \rho(a - b)$. We denote the monoid ring of $M$ over $R$ by $R[M]$.

Using the notation of  \cite{AM.CWTR} for all $a \in M$ we define $\tau_a$ to be the element of $R^M$ with $\tau_a(a) = 1$ and $\tau_a (M\backslash\{a\}) = \{0\}$. We observe that for all $f \in R[M]$ there is an $\mathbfsl{r} \in R^M$ such that $f = \sum_{a\in M} r_a\tau_a $ and that we can multiply such expressions using the rule $\tau_a \cdot \tau_b = \tau_{a+b }$. 

\begin{Def}\label{DefAct} Let $\FF_p$ and $\FF_q$ be finite fields and let $\FF_q^{\times}= (\FF_q, \cdot)$ be the multiplicative monoid reduct of $\FF_q$. We define the action $\ast: \FF_p[\FF_q^{\times}] \times \FF_p^{\FF_q} \rightarrow \FF_p^{\FF_q}$ for all $a \in \FF_q^{\times}$ and $f \in \FF_p^{\FF_q}$ by
	
	\begin{center}
		$(\tau_a \ast  f)(x) = f(ax).$
	\end{center}
	So for $\sigma \in \FF_p[\FF_q^{\times}] $ with $\rho = \sum_{a\in \FF_q^{\times}} z_a \tau_a $, then
	
	\begin{center}
		$(\sigma  \ast  f)(x) = \sum_{a\in \FF_q^{\times}} z_a f(ax).$
	\end{center}
\end{Def}

We can observe that $V$ is an $(\FF_p[\FF_q^{\times}],\ast)$-submodule of $\mathbb{F}_p^{\mathbb{F}_q}$ if and only if it is a subspace of $\mathbb{F}_p^{\mathbb{F}_q}$ satisfying

\begin{equation}
\label{neweq1}
x \mapsto f(ax) \in V,
\end{equation}
for all $f  \in V$ and $a \in \mathbb{F}_q$. Clearly the following lemma holds.

\begin{Lem}
Let $p$ and $q$ be powers of primes. Then the unary part of an $(\mathbb{F}_p , \mathbb{F}_q)$-linearly closed clonoid is an $(\FF_p[\FF_q^{\times}],\ast)$-submodule of $\mathbb{F}_p^{\mathbb{F}_q}$.
\end{Lem}

In order to show the following results we use the Galois correspondence between clonoids and pairs of relations as developed in \cite{Pip.GTFM}.
\begin{Def}
	\label{Defpol}
	For a set $I$ and $R \subseteq A^I$ , $S \subseteq B^I$ let	$\Pol(R, S) := \{f : A^k \rightarrow B : k \in \NN, f(R, \dots, R) \subseteq S\}$ denote the set of finitary functions preserving $(R, S)$. We call $\Pol(R, S)$ the set of \emph{polymorphisms} of the relational pair $(R, S)$.
\end{Def}

Let $R := \{(S_i , T_i ) : i \in I\}$ be a set of pairs of relations on $A$ and $B$. Then the set of functions that are polymorphisms of all pairs of all the relations in $R$ is denoted by $\Pol(R)$.

\begin{Lem}
\label{LemPoly}
Let $p$ and $q$ be powers of distinct primes. Let $U$ be the subspace of $\mathbb{F}_q^{\mathbb{F}_q}$ that is generated by the identity map on $\mathbb{F}_q$, and let $V$ be an $(\FF_p[\FF_q^{\times}],\ast)$-submodule of  $\mathbb{F}_p^{\mathbb{F}_q}$. Then $\Pol(U, V )$ is an $(\mathbb{F}_p,\mathbb{F}_q)$-linearly closed clonoid with unary part $V$.
\end{Lem} 

\begin{proof}
	
	By Definition \ref{Defpol} $V \subseteq \Pol(U, V )$ and every unary function in $\Pol(U, V )^{[1]} $ is in $V$. Thus $\Pol(U, V )^{[1]} = V$.
	
	Next we show that $\Pol(U, V )$ is an $(\mathbb{F}_p,\mathbb{F}_q)$-linearly closed clonoid. Let $f, g \in \Pol(U, V )$ be $m$-ary. Thus the functions $x \mapsto  f(c_1x, \dots, c_mx)$ and $x \mapsto g(c_1x, \dots, $ $c_mx)$ are in $V$ for all $c_1 ,\dots, c_m \in \mathbb{F}_q$. 
	
	Let $a, b \in \mathbb{F}_p$ and let $b_1,\dots,b_n \in \mathbb{F}_q$. Then
	
	\begin{center}			
		$(af + bg)(b_1x , \dots, b_mx) = af(b_1x , \dots, b_mx) + bg(b_1x , \dots, b_mx),$
	\end{center}
	for all $x \in \FF_q$. Hence $af + bg \in \Pol(U, V )$.	
	
	Next let  $n \in \NN$, let $A \in \mathbb{F}_q^{m\times n}$, and let $b_1,\dots,b_n \in \mathbb{F}_q$. Then 
	
	\begin{center}
		$x \mapsto f(A \cdot (b_1x,\dots, b_nx)^t ) = f( \sum_{i = 1}^nA_{1i}b_ix,\dots , \sum_{i = 1}^nA_{mi}b_ix)$
	\end{center}	
	is in $V$. Hence the $n$-ary function $g: \FF_q^n \rightarrow \FF_p$ such that $g: (x_1,\dots,x_n) \mapsto f(A \cdot (x_1,\dots, x_n)^t )$ is in $\Pol(U,V)$ and thus $\Pol(U,V)$ is an $(\mathbb{F}_p,\mathbb{F}_q)$-linearly closed clonoid.
\end{proof}

A clonoid $C$ with source set $A$ and target algebra $\mathbf{B}$ is \emph{finitely related} if there exists a finite set of pairs of finitary relations $R$ on $A$ and $B$ such that $C = \Pol(R)$. It can be easily observed that every finitely related clonoid is the clonoid of polymorphisms of a single pair of relations.

Together with Theorem \ref{Th1}, Lemma \ref{LemPoly} implies immediately the following.

\begin{Lem}
Let $p$ and $q$ be powers of distinct primes. Then every $(\mathbb{F}_p , \mathbb{F}_q )$-linearly closed clonoid is finitely related.
\end{Lem}

\begin{Cor}
	\label{CorModuleIso}
	Let $p$ and $q$ be powers of distinct primes. Then the function $\pi^{[1]}$ that sends an $(\mathbb{F}_p , \mathbb{F}_q)$-linearly closed clonoid to its unary part is an isomorphism between the lattice of all $(\mathbb{F}_p , \mathbb{F}_q)$-linearly closed clonoids and the lattice of all $(\FF_p[\FF_q^{\times}],\ast)$-submodules of $\mathbb{F}_p^{\mathbb{F}_q}$.
\end{Cor}

With the next lemma we begin to characterize the unary parts of the $(\FF_p,\FF_q)$-linearly closed clonoids. In order to do so we consider the group ring $\FF_p[\FF_q\backslash\{0\}]$ and the $(\FF_p[\FF_q\backslash\{0\}],\ast)$-submodules of $\mathbb{F}_p^{\mathbb{F}_q\backslash\{0\}}$ where $\ast$ is the restriction of the action defined in Definition \ref{DefAct} and $\FF_q\backslash\{0\}$ is the multiplicative group of $\FF_q$.

\begin{Lem}
	\label{LemUnaryIso}
	Let $p$ and $q$ be powers of distinct primes. Then the lattice of all $(\FF_p[\FF_q^{\times}],\ast)$-submodules of    $\mathbb{F}_p^{\mathbb{F}_q}$ is the direct product of the lattice of all $(\FF_p[\FF_q\backslash\{0\}],$ $\ast)$-submodules of $\mathbb{F}_p^{\mathbb{F}_q\backslash\{0\}}$ $\mathcal{L}_0(p,q)$ and the two-element chain $\mathbf{2}$.
\end{Lem}

\begin{proof}
	
	Let $V_0, V_1 \subseteq \mathbb{F}_p^{\FF_q}$ be defined by:
	
	\begin{flushleft}
		$V_0 := \{\mathbfsl{v} \in \mathbb{F}_p^{\FF_q}\mid v_0=0\}$
		\\$V_1 := \{\lambda(1,\dots,1)\mid \lambda \in \mathbb{F}_p\}$.
	\end{flushleft}
	It is clear that $V_0$ and $V_1$ are $(\FF_p[\FF_q^{\times}],\ast)$-submodules of $\mathbb{F}_p^{\mathbb{F}_q}$ and we denote  the lattices of all $(\FF_p[\FF_q^{\times}],\ast)$-submodules of $V_0$ and $V_1$ by $\mathcal{L}_0$ and $\mathcal{L}_1$. We can observe that $\mathcal{L}_1 \cong \mathbf{2}$. Moreover, $V_0 \vee V_1 = 1$ and $V_0 \cap V_1 = 0$. Let $W$ be an $(\FF_p[\FF_q^{\times}],\ast)$-submodule of $\mathbb{F}_p^{\mathbb{F}_q}$. Then we have that either $W \leq V_0$ or $W \geq V_1$, since $\mathbfsl{v} \not\in V_0$ implies $v_0 \not=0$ and thus that $(1,\dots,1)$ is in the $(\FF_p[\FF_q^{\times}],\ast)$-submodule of $\mathbb{F}_p^{\mathbb{F}_q}$ generated by $\mathbfsl{v}$. Next, we show that $W = (W \cap V_0) \vee (W \cap V_1)$:
	
	Case $W \leq V_0$: then $(W \cap V_0) \vee (W \cap V_1) = W  \vee (W \cap V_1) = W$
	
	Case $W \geq V_1$: then $(W \cap V_0) \vee (W \cap V_1) = (W \cap V_0) \vee V_1$ which is equal to $W \cap (V_0 \vee V_1) = W$, using the modular law.
	Thus the function $\gamma$ from the lattice $\mathcal{L}_0 \times \mathcal{L}_1$ to  the lattice of all $(\FF_p[\FF_q^{\times}],\ast)$-submodules of $\mathbb{F}_p^{\mathbb{F}_q}$ such that $\gamma: (R,W) \mapsto R \vee W$ is clearly surjective and order preserving. Furthermore, let $(R_1,W_1), (R_2,W_2) \in \mathcal{L}_0 \times \mathcal{L}_1$ with $R_1 \vee W _1 = R_2 \vee W_2$. Then, since $R_1, R_2 \leq V_0$, $W_1, W_2 \leq V_1$, and $V_0 \cap V_1 = 0$, $R_1 = (R_1 \vee W_1) \cap V_0 = (R_2 \vee W_2) \cap V_0 = R_2$, using the modular law. With the same strategy we can prove that $W_1 = W_2$ and it is clear that $\gamma^{-1}$ is order preserving. Thus $\gamma$ is a lattice isomorphism.

	 Hence, by Corollary \ref{CorModuleIso}, $\mathcal{L}(p,q)$ is isomorphic to $\mathbf{2} \times \mathcal{L}_0$. Furthermore, the lattice $\mathcal{L}_0$  is isomorphic to $\mathcal{L}_0(p,q)$ via the isomorphism $\pi_{\geq 2} : \mathcal{L}_0 \rightarrow \mathcal{L}_0(p,q)$ such that:
	\begin{equation}\label{pip}
	\pi_{\geq 2}(C):= \{(v_1,\dots,v_{q-1}) \mid (0,v_1,\dots,v_{q-1}) \in C\},
	\end{equation}
	for all $C \in \mathcal{L}_0$.
\end{proof}

 The next step is to characterize the lattice $\mathcal{L}_0(p,q)$. To this end we observe that $V \in \mathcal{L}_0(p,q)$ if and only if is a subspace of $\FF_p^{\FF_q\backslash\{0\}}$ satisfying 
 
\begin{equation}
\label{neweq2}
x \mapsto f(ax) \in V,
\end{equation}
for all $f  \in V$ and $a \in \mathbb{F}_q\backslash\{0\}$. 

Let $\alpha$ be a generator of the multiplicative subgroup of $\FF_q$. The closure under the linear transformation $A(p,q): \mathbb{F}_p^{\FF_q \backslash\{0\}} \to  \mathbb{F}_p^{\FF_q \backslash\{0\}}$ defined as $A(p,q)(f)= x \mapsto f(\alpha x)$ is enough to describe the property to be closed with all the $q-1$ linear transformations in (\ref{neweq2}). 
We can see that the minimal polynomial of $A(p,q)$ is $x^{q-1}-1$.

Thus the last step is to characterize the lattice of the $A(p,q)$-invariant subspaces of $\mathbb{F}_p^{\FF_q \backslash\{0\}}$ which is isomorphic to the lattice of all $(\FF_p[\FF_q\backslash\{0\}],$ $\ast)$-submodules of $\mathbb{F}_p^{\mathbb{F}_q\backslash\{0\}}$ $\mathcal{L}_0(p,q)$ and to the lattice of all $(\FF_p,\FF_q)$-linearly closed clonoids composed by $0$-preserving functions.

Corollary \ref{CorModuleIso} and Lemma \ref{LemUnaryIso} describe the structure of the lattice of all $(\mathbb{F}_p,\mathbb{F}_q)$-linearly closed clonoids in case $p$ and $q$ are powers of distinct primes. In the figure below we draw a scheme of the lattice of all $(\mathbb{F}_p,\mathbb{F}_q)$-linearly closed clonoids. On the right hand side we have the $0$-preserving part and on the left, the part with constants. Let $\mathbf{1},\mathbf{0_P}, \mathbf{C}, \mathbf{\{0\}}$ denote the $(\mathbb{F}_p,\mathbb{F}_q)$-linearly closed clonoids of all functions, of all $0$-preserving functions, of all constants, and of the zero constants respectively.
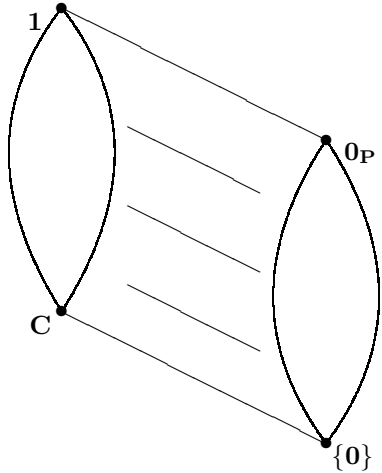
\vspace*{4.0cm}
\begin{figure}[htbp]
	\begin{center}
		\begin{picture}(50,80)

		\put(50,130){\circle*{4}}
		\put(50,15){\circle*{4}}
		\put(-50,180){\circle*{4}}
		\put(-50,65){\circle*{4}}
		
		\put(63,125){\makebox(0,0){\footnotesize$\mathbf{0_P}$}}
		\put(60,10){\makebox(0,0){\footnotesize$\mathbf{\{0\}}$}}		
		\put(-60,175){\makebox(0,0){\footnotesize$\mathbf{1} $}}
		\put(-58,60){\makebox(0,0){\footnotesize$\mathbf{C}$}}
		
		\qbezier(50,15)(90,69)(50,130)	
		\qbezier(50,15)(10,69)(50,130)	
		\qbezier(-50,65)(-90,128)(-50,180)	
		\qbezier(-50,65)(-10,128)(-50,180)		
		\put(50,15){\line(-2,1){100}}
		\put(50,130){\line(-2,1){100}}
		\put(25,50){\line(-2,1){50}}
		\put(25,80){\line(-2,1){50}}
		\put(25,110){\line(-2,1){50}}
		
		\end{picture}
	\end{center}
\end{figure}

The lattice of invariant subspaces under a linear transformation on finite-dimensional vector spaces was characterized in \cite{BF.TISL}. We can see that $A(p,q)$ has minimal polynomial $g = x^{q-1} -1$. Let $g = \prod_{i =1}^s p_i^{k_i}$ be the prime factorization of $g$ over $\mathbb{F}_p[x]$. We define $V_i = ker(p_i(A(p,q))^{k_i})$, and $A(p,q)_i = A(p,q)|_{V_i}$. We know from \cite{BF.TISL} that with the prime factorization of $g$ we can split our vector space $\mathbb{F}_q^{\FF_q\backslash\{0\}}$ into its primary decomposition:

\begin{center}
	$\mathbb{F}_q^{\FF_q\backslash\{0\}} =\bigoplus_{i = 1}^s V_i$.
\end{center}
According to \cite{BF.TISL} the lattice $\mathcal{L}(A(p,q))$ of the  $A(p,q)$-invariant subspaces of $\mathbb{F}_q^{\FF_q\backslash\{0\}}$, is:

\begin{equation}\label{eqdirprod}
	\mathcal{L}(A(p,q)) \cong \prod_{i = 1}^{s}\mathcal{L}(A(p,q)_i).
\end{equation}
We can observe that $\mathbb{F}_q^{\FF_q\backslash\{0\}}$ is an $A(p,q)$-cyclic space generated by $(1,0,\dots,0)$. Thus, every $A(p,q)$-invariant subspace of $\mathbb{F}_q^{\FF_q\backslash\{0\}}$ is $A(p,q)$-cyclic, by Remark \ref{Rem}.

With these tools we are now ready to prove Theorem \ref{Thm14} and \ref{ThmShapeL1}.

\begin{proof}[Proof of Theorem \ref{Thm14}.]
	Let $C$ be an $(\mathbb{F}_p,\mathbb{F}_q)$-linearly closed clonoid and let $C_0$ be its $0$-preserving part. Let $V$ be the image of $C_0$ under the isomorphism of Corollary \ref{CorModuleIso} and let $(V_1,V_2)$ be the image of $V$ under the isomorphism of Lemma \ref{LemUnaryIso}. Furthermore let $W$ be the image of $V_1$ under $\pi_{\geq 2}$. We have observed that $W$ is an $A(p,q)$-cyclic space. Let $\mathbfsl{v}$ be the $A(p,q)$-cyclic vector for $W$ (Definition \ref{DefCyclic}). We can observe that $f = (0,v_1,\dots,v_{q-1})$ is a generator for $C_0^{[1]}$ and thus, by Theorem \ref{Th1}, is a unary generator for $C_0$. Furthermore, either $C_0 = C$ or $C_0 \subset C$.
	
	Case $C_0 = C$: then $C$ is generated by $f$.
	
	Case $C_0 \subset C$: then $\{1\} \in C$, where $1$ is the constant unary function with value $1$. Let $f$ be the unary generator of $C_0$. We will prove that $C$ is generated by $h = f + 1$. Indeed, let $g \in C$ be an $n$-ary function. Then, there exists a $0$-preserving $n$-ary function $g_0$ such that $g = g_0 + g(\mathbfsl{0})$, where $g(\mathbfsl{0})$ is the constant $n$-ary function with value $g(\mathbfsl{0})$. Hence, $g \in \Cid(\{g_0\}) \vee \Cid(\{g(\mathbfsl{0})\}) \subseteq \Cid(\{f\}) \vee \Cid(\{1\}) \subseteq \Cid(\{h\})$. Thus $C = \Cid(\{h\})$ and the claim holds.
	
\end{proof}

\begin{proof}[Proof of Theorem \ref{ThmShapeL1}]
	Let $p$ and $q$ be powers of distinct primes and let $\prod_{i =1}^n p_i^{k_i}$ be the prime factorization of the polynomial $g = x^{q-1} -1$ in $\mathbb{F}_p[x]$. First we know from Corollary \ref{CorModuleIso} and Lemma \ref{LemUnaryIso} that $\mathcal{L}(p,q) \cong \mathbf{2} \times \mathcal{L}_0(p,q)$. Furthermore, we know that $\mathcal{L}_0(p,q)$ is isomorphic to the lattice of all $A(p,q)$-invariant subspaces of $\FF_p^{\FF_q\backslash\{0\}}$. We have observed that $A(p,q)$ has $g$ as minimal polynomial. So let $V_i$ be the $i$th primary component and let $(A(p,q))_i$ be the $i$th restriction of $A(p,q)$ with minimal polynomial $p_i^{k_i}$, for $i=1,\dots,n$. Then we know that $\mathbb{F}_p^{\FF_q\backslash\{0\}}$ is $(A(p,q))$-cyclic with $(1,0,\dots,0)$ as $(A(p,q))$-cyclic vector. Hence also $V_i$ is $(A(p,q))$-cyclic, for $i=1,\dots,n$, as a subspace of an $(A(p,q))$-cyclic space (Remark \ref{Rem}). From \cite[Lemma $2$]{BF.TISL}, the lattice of all $A(p,q)_i$-invariant subspaces of $V_i$ is isomorphic to the chain with $k_i + 1$ elements. 
	Thus, from (\ref{eqdirprod}), we have that:	
	\begin{center}
		$\mathcal{L}(p,q) \cong \mathbf{2} \times \prod_{i =1}^n \mathbf{C}_{k_i+1}$
	\end{center}
and the claim holds.

\end{proof}

With this theorem we have completely characterized the structure of the lattice $\mathcal{L}(p,q)$ using the prime factorization of the polynomial $x^{q-1}-1$, which can be easily computed.
We conclude our investigation with a corollary that shows how the lattice of all $(\mathbb{F}_p,\mathbb{F}_q)$-linearly closed clonoids is structured.

\begin{Cor}
	Let $p$ and $q$ be powers of distinct primes. Then the lattice $\mathcal{L}(p,q)$ of the $(\mathbb{F}_p,\mathbb{F}_q)$-linearly closed clonoids is a distributive lattice.
\end{Cor}

\begin{proof}
	It follows from Theorem \ref{ThmShapeL1}  that $\mathcal{L}(p,q)$ is a direct product of chains and hence is distributive.
\end{proof}

\section*{Acknowledgements}

The author thanks Erhard Aichinger, who inspired this paper, and Sebastian Kreinecker for many hours of fruitful discussions. The author thanks the referees for their useful suggestions.

\bibliographystyle{alpha}
\end{document}